\documentclass[10pt]{amsart} 

\usepackage{amsmath,amstext,amscd, amssymb,txfonts, epsfig, psfrag, color, epstopdf} 

  \usepackage{tikz}
\usetikzlibrary{decorations.markings}
\usetikzlibrary{decorations.pathreplacing}
\usetikzlibrary{arrows.meta}
\usetikzlibrary{calc}
\usetikzlibrary{patterns}
  \usepackage{tkz-euclide} 

\usepackage{thmtools, thm-restate}

\newtheorem{thm}{Theorem}[section]
\newtheorem{lemma}[thm]{Lemma}
\newtheorem{proposition}[thm]{Proposition}

\newtheorem{letterthm}{Theorem}

\newtheorem{lettercor}[letterthm]{Corollary}

\theoremstyle{definition} 
\newtheorem{definition}[thm]{Definition} 
\newtheorem{remark}[thm]{Remark}  
\newtheorem{remarks}[thm]{Remarks}

\numberwithin{equation}{section}

\def\vi{v_\infty}

\def\Q{\mathfrak{Q}}

\def\Z{\mathbb{Z}}
\def\N{\mathbb{N}}

\def\diam{\hbox{\rm diam}}

\def\l{\lambda}

\def\onto{{\kern3pt\to\kern-8pt\to\kern3pt}}

\def\<{\langle}
\def\>{\rangle}
\def\|{{\ |\ }}
\def\iff{{\Leftrightarrow}}

\def\cay{{\rm{Cay}}}

\def\G{\Gamma}
\def\d{\delta}

\def\l{\lambda}
\def\e{\varepsilon}

\def\*{^{\star}}

\begin{document}

\title{On the conjugacy problem for subdirect products of hyperbolic groups}

\author[Martin R.   Bridson]{Martin R.  ~Bridson}
\address{Mathematical Institute\\
Andrew Wiles Building\\
Woodstock Road\\
Oxford,   OX2 6GG}
\email{bridson@maths.ox.ac.uk}




\begin{abstract}  If $G_1$ and $G_2$ are torsion-free hyperbolic groups and $P<G_1\times G_2$
is a finitely generated subdirect product, then  the conjugacy problem in $P$
is solvable if and only if there is a uniform algorithm to decide membership of the cyclic
subgroups in the finitely presented group $G_1/(P\cap G_1)$.
The proof of this result
relies on a  new technique for perturbing elements in a hyperbolic group to ensure that they are not proper powers. 
\end{abstract}

\maketitle

\def\ov{\overline}
\def\-{overline}
\def\tP{\tilde{P}}
\def\e{\varepsilon}
\def\o{\mathfrak{o}}

\section{Introduction}
A subgroup $P<G_1\times\cdots\times G_m$ is a {\em subdirect product} if its projection to each  
 factor $G_i$ is onto. 
In this article we will be concerned with the conjugacy problem for subdirect products of torsion-free 
hyperbolic groups.   In this setting, 
Bridson, Howie, Miller and Short \cite{BHMS} proved that
if $P$ projects to a subgroup
of finite index in each pair of factors $G_i\times G_j$, then $P$ is finitely
presented 
and its conjugacy problem is solvable.   
If $P$ does not virtually surject to each $G_i\times G_j$, then the situation is
much wilder.

In the  case $m=2$, there
is a correspondence between subdirect products and {\em fibre products}: if $P<G_1\times G_2$ is subdirect
then there is an isomorphism
$G_1/(P\cap G_1)\cong G_2/(P\cap G_2)$ and $P$ is the
fibre product of  the maps $p_i:G_i\onto G_i/(P\cap G_i)$;
see \cite{BM}, where it is proved that $P$ is finitely generated if and only if $Q:=G_1/(P\cap G_1)$
is finitely presented.

The case where $G_1\cong G_2\cong F$ is a free group was studied by Mihailova \cite{mihailova} and 
Miller \cite{cfm:thesis}. They proved that if $p_1=p_2$ and $Q$ has an unsolvable word problem, 
then neither the membership problem nor the conjugacy problem
for $P$ can be solved algorithmically.  These fibre products
$P<F\times F$ are not finitely presented,  and for finitely presented subdirect products of free groups (with
any number of factors),  the conjugacy  and membership problems can always be solved \cite{BM}.
But finite presentation does not save us in the hyperbolic setting: Baumslag,  Bridson, Miller and Short \cite{BBMS} constructed torsion-free hyperbolic groups $G$ and finitely presented  fibre
products $P<G\times G$ for which both the membership and conjugacy problems are unsolvable.

Our purpose here is to strengthen these results by establishing
criteria for solvability that are both necessary and sufficient.  It is not difficult to show that
the membership problem for a finitely generated
subdirect product of hyperbolic groups $P<G_1\times G_2$ is solvable if and only if the word problem in 
$Q= G_1/(G_1\cap P)$ is solvable or, equivalently, 
the Dehn function of $Q$ is recursive (see Theorem D in \cite{mb:CLfib} for a quantified version of this). 
In search of a similar characterisation for the solvability of the  
conjugacy problem in $P$, we are forced to examine the membership problem for cyclic subgroups
in $Q$.   

It is traditional to call the uniform membership problem for cyclic subgoups  the power problem.
Thus a finitely generated group $Q$ has a {\em solvable power problem} if there is  an algorithm that,
given two words $u,v$ in the generators,  will decide whether or not $u$ lies in the subgroup
of $Q$ generated by $v$.  When one moves from consideration of the word problem to consideration of the
power problem, the   appropriate
analogue of the Dehn function  is the {\em rel-cyclics Dehn function}, which 
we introduce in Section \ref{s:dehn}.

\begin{restatable}{letterthm}{mainthm} \label{t:iff} Let  $G_1$ and $G_2$
be torsion-free hyperbolic groups,    let $P<G_1\times G_2$ be a finitely generated
subdirect product, and let $Q$ be the finitely presented group $G_1/(G_1\cap P)$.   Then,   
the following conditions are equivalent:
\begin{enumerate}
\item the conjugacy problem in $P$ is solvable;
\item the power problem in $Q$ is solvable;
\item  the rel-cyclics Dehn function $\delta_Q^c(n)$ is recursive.
\end{enumerate} 
\end{restatable}

It is important to note that although the power problem covers membership of both finite and infinite 
cyclic subgroups,  a solution to the power problem does not allow one to determine which elements of 
the group have finite order.  Indeed,  Collins \cite{collins} proved that the order problem, i.e. the problem of
deciding the orders of elements, can be arbitrarily harder.  This points to a subtle necessity in the definition of
the rel-cyclics Dehn function: it has to control the distortion of cyclic subgroups without determining whether those
subgroups are infinite. Collins also proved that the power problem can be arbitrarily
harder than the word problem (which is the membership problem for the trivial cyclic subgroup $\{1\}$).
In the light of these results,  the 
following corollary is a straightforward consequence of Theorem \ref{t:iff}.  

\begin{lettercor}\label{c:collins}
There exist torsion-free hyperbolic groups $G$ and finitely generated subdirect products $P<G\times G$
such that the membership problem for $P$ is solvable but the conjugacy problem for $P$ is not.   

There also exist examples where the membership and conjugacy problems for $P$ are solvable but
there is no algorithm to decide which elements of $G\times G$ have a non-zero power that lies in $P$.
\end{lettercor}

When trying to understand conjugacy problems, one is inevitably drawn into studying centralisers, because
if $\gamma$ conjugates $u$ to $v$ in a group $G$, then the set of all solutions to the equation $x^{-1}ux=v$
is the coset $C_G(u)\gamma$.  In torsion-free hyperbolic groups,  centralisers of non-trivial elements are cyclic. In our proof
of Theorem \ref{t:iff},  certain arguments only work smoothly when the element of $G$ under consideration  generates its
own centraliser. The following proposition will be useful in this regard (and I imagine it may be of
similar use in other settings).  A related (non-algorithmic) result is contained in Minasyan's work on conjugacy separability 
for subdirect products of hyperbolic groups \cite{minasyan}.
We will deduce Proposition \ref{p:no-roots}  from a more geometric result, Proposition \ref{t:no-power},  which in turn is
motivated by the Lyndon-Sch\"{u}tzenburger Theorem \cite{LS}, which states that in a non-abelian free group, if
$a,b,c$ do not pairwise commute, then  $a^pb^qc^r\neq 1$ for all powers  $p,q,r\ge 2$.

 \def\EE{\mathcal{E}}

\begin{restatable}{letterprop}{otherletterprop}\label{p:no-roots} 
Given a torsion-free hyperbolic group $G=\<X\mid R\>$,  a group  
$Q$ with a solvable word problem, and a non-injective epimorphism $p:G\onto Q$, there is
a finite set $\EE\subset F(X)$ and an algorithm that, given a word $w\in F(X)$,  
will output a word   $w'\in F(X)$ with $p(w)=p(w')$ in $Q$ such that $w'\in\EE$ or else
the centraliser $C_G(w')$ is $\<w'\>$.
\end{restatable}

Throughout this article, it will be helpful to take the formal viewpoint that a choice of generators $X$ for a 
group $G$ is an epimorphism  $F(X)\onto G$ from the free group on $X$.  
To prevent a clutter of notation,  
we avoid giving this epimorphism a name: instead,  we use the phrase ``in $G$" when there is a need
to specify that words in the generators are being considered as products of the generators
in $G$ rather than elements of $F(X)$.   Similarly,  we write ``$u=v$ in $G$" if $uv^{-1}\in\ker(F(X)\onto G)$, and we write $C_G(w)$
to denote the centraliser in $G$ of the image of $w\in F(X)$.   
An advantage of this convention is that, given
an epimorphism $p:G\onto Q$,  one can regard $X$ as a generating set for $Q$ by composing  $F(X)\onto G$ 
with $p:G\onto Q$.   
From a geometric viewpoint, this is useful because it allows us to extend $p:G\onto Q$ to a label-preserving local isometry 
of Cayley graphs $\cay(G,X)\onto \cay(Q,X)$. In the associated word metrics,  $d_Q(p(h),p(g)) \le d_G(h,g)$ for 
all $h,g\in G$.   

The elements $h\in G$ in the following proposition are those that are represented by words $w\in F(X)$
that are geodesics in $\cay(Q_a,X)$.

\begin{restatable}{letterprop}{firstletterprop}{\rm{[Power-Avoiding Lemma]}}\label{t:no-power}  
Let $G$ be  a hyperbolic group with finite generating set $X$ and consider $X$
as a generating set for each of the quotients $Q_a:=G/\<\!\< a\>\!\>$.   Then, 
for each  element of infinite order $a\in G$,  there exist constants $K,N$ with the following property: 
if   $d_G(1,h) = d_{Q_a}(\<\!\< a\>\!\>,h\<\!\< a\>\!\>) >N$, then $ha^K\in G$ is not a proper power.
\end{restatable}

The lengthy  proof of Proposition \ref{t:no-power} is  presented in Section \ref{s:no-roots}. I hope that this
result may be of independent interest, but the reader who is willing to take Propositions \ref{p:no-roots} and
\ref{t:no-power}  on faith
can move directly from the preliminaries in Sections \ref{s:dehn} and \ref{s:prelim} to the proof of Theorem \ref{t:iff}
in Section \ref{s:proof}.

\noindent{{\bf{Acknowledgements.}} 
I am grateful to the University of Oxford for the sabbatical leave that enabled me to finish this article,
and to Stanford University for hosting me during this sabbatical. I am also grateful to Tim Riley for extensive and fruitful
conversations about the complexity of conjugacy problems,   to Ashot Minasyan and Denis Osin for 
helpful correspondence regarding analogues of the Lyndon-Sch\"{u}tzenburger Theorem in hyperbolic groups, and to
the anonymous referee for their careful reading and helpful comments.
My special thanks go to Chuck Miller for the many insights that I gained
from him during our long collaboration on fibre products and decision problems.

\section{The rel-cyclics Dehn function}\label{s:dehn}

Let $\mathcal{P} \equiv \< X \mid R\>$ be a finite presentation of a group $Q$. By definition,
a word $w$ in the free group $F(X)$ represents the identity in $Q$ if and only if 
there is an equality in $F(X)$ 
\begin{equation}\label{w:equ}
w = \prod_{i=1}^M \theta_i^{-1} r_i \theta_i
\end{equation}
with $r_i\in R\cup R^{-1}$ and $\theta_i\in F(X)$.  
The number $M$ of factors in the product on the righthand side is defined to be 
the {\em{area}} of the product.
This terminology is motivated by the   correspondence with diagrams
that comes from van Kampen's Lemma \cite{mrb:bfs}.

One defines
${\rm{Area}}(w)$ to be the least area among all products of this form for $w$.
The {\em{Dehn function}} of $\mathcal{P}$ is the function $\delta: \N \to \N$ defined by
$$
\delta(n) = \max \{{\rm{Area}}(w) \mid w=_Q 1 \ {\rm{and}} \ |w|\le n\}.
$$
Up to the standard coarse bi-Lipschitz equivalence relation $\simeq$ of geometric group theory,
this function is independent of the chosen presentation and it is common to abuse notation 
by adding a subscript, viz.~$\delta_Q(n)$.

The {\em noise} of the product  on the right of (\ref{w:equ}) is defined to be 
$\sum_{i=0}^{M} |\theta_{i}\theta_{i+1}^{-1}|$, with the convention that 
$\theta_0$ and $\theta_{M+1}$ are the empty word.
An analysis of the standard proof of van Kampen's Lemma  
yields the {\em Bounded Noise Lemma}, an observation (and terminology) that is
due to the authors of \cite{BBMS}; see \cite{dison}, p.18
for an explicit proof.

\begin{lemma}[Bounded Noise Lemma]\label{l:BNL}
If ${\rm{Area}}(w) = M$, then there is a product of the form (\ref{w:equ})
with area $M$ and noise at most $ML + |w|$, where $L$ is the length of the longest relator in $R$.
\end{lemma}

\begin{remark}\label{r:solve}
With the Bounded Noise Lemma in hand, it is clear that the word problem is solvable in $Q$ if and only
if $\delta_Q(n)$ is a recursive function: on the one hand,  knowing that $w=1$ in $Q$, one can naively
search for equalities of the form (\ref{w:equ}) until one is found,  giving a bound on 
${\rm{Area}}(w)$ that can be sharpened to an exact value by checking the finitely many products with
fewer factors that satisfy the Bounded Noise Lemma -- thus $\delta_Q(n)$ can be computed; conversely, 
given a word  $w$ of length $n$,  if $\delta_Q(n)$ can be computed, then one can determine whether $w=1$ in $Q$
by testing the validity of  (\ref{w:equ}) for each product with $M\le \delta_Q(n)$ that satisfies the Bounded Noise Lemma, and there are only finitely many such equalities that need to be tested.

We shall employ a similar argument to see that the recursiveness of the following function is equivalent to the
solvability of the power problem in $Q$.
\end{remark}

\begin{definition}\label{d:d^c}
The {\em rel-cyclics} Dehn function of a finitely presented group $Q=\<X\mid R\>$ is 
$$ 
\d^{c}(n) : = \max_{w,u} \{ {\rm{Area}}(w\,u^{-p}) +|pn| \colon |w|+|u|\le n,\ w=_Q u^{p},\
|p|\le  o(u)/2\},
$$
where $o(u)\in\N\cup\{\infty\}$ is the order of $u$ in $\G$.  
\end{definition}

\begin{remarks} 
(1) The condition $|p|\le  o(u)/2$
is equivalent to requiring that $|p|$ is the least non-negative  integer such that  $w=u^{\pm p}$ in $Q$.
  
(2) We allow $p=0$, so $ \d(n) \le \d^{c}(n)$.

(3) The standard proof that the Dehn function of a finitely presented group is 
independent of the presentation, up to $\simeq$ equivalence,  shows that $\delta^c(n)$ is as well.  
With this understanding,  we write $\d_Q^{c}(n)$ when we want to emphasize which group we are considering.

(3)  If we replaced the condition  $|p|\le o(u)/2$ with $|p|\le o(u)$ then  the resulting variant of $\d^c(n)$
would control the orders of $u\in Q$ with $|u|\le n$; in particular the order problem in $Q$ would be solvable
when this function was recursive.  The work of Collins \cite{collins} and
Corollary \ref{c:collins} tell us that we must avoid this in the context of Theorem \ref{t:iff}.  
\end{remarks}  

The following proposition establishes the equivalence of the second and third conditions in Theorem \ref{t:iff}.

\begin{proposition}\label{ThmA:2iff3}
The power problem in a finitely presented group $Q$ is solvable if and only if the rel-cyclics Dehn function 
$\delta_Q^c(n)$ is recursive.
\end{proposition}

\begin{proof} First we assume that the power problem in $Q$ is solvable.  
For each pair of words $u,v$ of length at most $n$ in the generators of $Q$,  we use the solution to the
power problem to decide if $v$ lies in the cyclic subgroup of $Q$ generated by $u$. If it does not, we discard this pair.
If it does, then we use the solution to the word problem in $Q$ (i.e.~the membership problem for the cyclic group $\{1\}$)
to test the words $vu^i$, with $|i|$ increasing,  until we identify
the least non-negative integer $|p|$ such that $v=u^p$ or $v=u^{-p}$ in $Q$; replacing $p$ by $-p$ if necessary,
we may assume  $v=u^{p}$.
 Note that if $u$ has finite order then the
 minimality of $|p|$ implies $|p|\le o(u)/2$.  With $p$ in hand,  we can calculate  ${\rm{Area}}_Q(vu^{-p})$
 by searching for equalities of the form (\ref{w:equ})  with $w=vu^{-p}$, restricting our attention to those
 that satisfy
 the Bounded Noise Lemma (Lemma \ref{l:BNL}): first, a naive search will identify one valid equality, giving an
 upper bound $M$ on ${\rm{Area}}_Q(vu^{-p})$, and the exact value of ${\rm{Area}}_Q(vu^{-p})$ can then be
 calculated by checking the validity of (\ref{w:equ}) for the finitely many products with at most $M$ factors that satisfy the
 Bounded Noise Lemma.  Having calculated ${\rm{Area}}_Q(vu^{-p})$ for each
 pair $u,v$ and $|p|$  minimal, we obtain  $\delta_Q^c(n)$.   Thus $\delta_Q^c(n)$ is a recursive function.
 
\medskip
Conversely,  if the rel-cyclics Dehn function is recursive,  then, given  words $u,v$, we take an integer $n$
that is greater than the length of each and compute 
 $\delta_Q^c(n)$.  If $v=u^p$ for some $p\in\Z$, then for the minimal $|p|$ we have $|p|n\le \delta_Q^c(n)$.
Moreover,  when $|p|$ is minimal,  ${\rm{Area}}(vu^{-p})\le \delta_Q^c(n)$. Thus, in order to determine whether
$v\in\< u \>$,  we  only need to check for each $p$ with $|p|\le \delta_Q^c(n)$ whether there is a valid
free quality as in (\ref{w:equ}), with $w=vu^{-p}$,  running over only products with $M\le \delta_Q^c(n)$ factors
that satisfy the conclusion of the Bounded Noise Lemma. This is a finite check. 
\end{proof}

\section{Some basic hyperbolic facts}\label{s:prelim}

The reader is assumed to be familiar with the rudiments of Gromov's theory of hyperbolic groups \cite{gromov}, but it is nevertheless useful to gather  the following basic facts and settle on names for various constants. 
These are mostly standard facts from \cite{BH}, pp.~402--406,  with the exception of Lemma \ref{l:constants}(3). As is standard,
we say that $G$ is $\delta$-hyperbolic if each side of every geodesic triangle in the Cayley graph $\cay(G,X)$
is contained in the $\delta$-neighbourhood of the union of the other two sides. 

\begin{lemma}\label{l:constants}
For any $\delta$-hyperbolic group $G$ with finite generating set $X$,  and for all $\l\ge 1,\, \e\ge 0$,
there are positive constants $E, C$ and $\mu$ such that  the following statements hold.
\begin{enumerate}
\item  The Hausdorff distance between  any $(\l,\e)$-quasigeodesic in $\cay(G,X)$
and any geodesic with the same endpoints is at most $E$.
\item Every  quadrilateral in $\cay(G,X)$ with $(\l,\e)$-quasigeodesic sides is $C$-slim, 
meaning that each side is contained in the $C$-neighbourhood of the union of the other 3 sides.
\item For every $g\in G$ of infinite order,  $d(1,g^i) \le \mu\,  d(1,g^p)$ for all integers $0<i<p$.
\end{enumerate} 
\end{lemma}

\begin{proof}  (1) is the standard Morse Lemma for hyperbolic spaces,  \cite{BH}, p.~401. Item  (2) follows easily
from (1) and the slimness of triangles (after introducing a diagonal).  
Assertion (3)  is  proved in \cite[Proposition 4.7]{mb:CLfib}.
\end{proof}
 
 \begin{lemma}\label{l:local-geod}
For any $\delta$-hyperbolic group $G$ with finite generating set $X$,  
there exist positive constants $\e_0$ and $\l_0\ge 1$ such that  the following statements hold.
\begin{enumerate}
\item  For each  conjugacy class of  infinite-order elements in $G$,   if $w\in F(X)$ is a word of minimal length 
among all representatives of this class,  then the lines in $\cay(G,X)$ labelled $w^*$ are 
$(\l_0, \e_0)$-quasigeodesics;
\item moreover,  if $|w| > 12\delta$ then $d(1, w^p) \ge  \frac{p}{2}|w| -2\delta$ for all $p\ge 0$.
\end{enumerate} 
\end{lemma} 

\begin{proof}
If  $|w|>8\delta$, then according to Theorem III.H.1.13 of \cite{BH},  
the lines labelled $w^*$ are $(\l_0,\e_0)$-quasigeodesics where the
constants $\l_0, \e_0$ depend only on $\delta$.
There are only finitely many $w$ to consider with $|w| \le 8\d$,  so increasing the constants $\l_0, \e_0$ 
if necessary,  we may assume that the lines for these are also $(\l_0,\e_0)$-quasigeodesics.   
 Item (2)  is a special case of   \cite[Theorem III.H.1.13(3)]{BH}.
\end{proof} 

\begin{lemma}\label{l:find-root}
If $G$ is torsion-free and hyperbolic, then there is an algorithm that, given a word $w$ in the generators, will decide if 
$w$ is non-trivial in $G$ and, if it is,  will produce a word $w_0$ such that $C_G(w) = \<w_0\>$.
\end{lemma}

\begin{proof}
The word problem in $G$ is solvable, so we can determine whether $w\neq 1$. If $w\neq 1$, then  $C_G(w)$
is cyclic.  Lemma \ref{l:constants}(3) tells us that $C_G(w)$ is generated by $g\in G$ with $d(1,g)\le \mu |w|$.
For each word $v$ of length at most $\mu |w|$, we use the solution to the word problem to decide if $[v,w]=1$
in $G$. If $[v,w]=1$ and $v\neq 1$ in $G$,  then, searching in parallel,  we use the solution to the word problem to either
find coprime integers $q\ge 1$ and $q'\ge 2$ such that $v^q=w^{\pm q'}$, in which case we discard $v$,  or else 
find an integer $p\ge 1$ such that $w= v^{\pm p}$.   The largest $p$ that is found identifies 
the corresponding $v$ as  $w_0$.
\end{proof}

\subsection{Thin triangles and tripod comparison}\label{s:slim}

In the proof of Proposition  \ref{t:no-power},  
the most convenient definition of hyperbolicity to work with is
the one phrased in terms of {\em{$\eta$-thin triangles}}. We recall this definition here
and refer the reader to  \cite{BH} pp.409-411 for details of how it is equivalent to other
hyperbolicity conditions.    
This formulation is particularly useful when one wants
to make arguments with several juxtaposed triangles or polygons, following small
fibres (point inverses of the maps $f_\Delta$ defined below) around the diagram.

By definition, a triple of positive numbers $l_1,l_2,l_3$ satisfies the triangle inequality
if $l_i+j_j\le l_k$ for all permutations $i,j,k$ of $1,2,3$. Associated to each such triple there
is a unique tripod $T(l_1,l_2,l_3)$ -- that is,  a metric graph with four vertices $c,v_1,v_2,v_3$
such that $d(v_i,v_{i+1})=d(c,v_i)+ d(c,v_{i+1}) = l_i$, with indices {\rm{mod}} $3$;
there is a permitted degeneracy $c=v_i$ when $l_{i-1} + l_i= l_{i+1}$.

Given a triangle $\Delta=\Delta(y_1,y_2,y_3)$ in any geodesic space $Y$, 
with $l_i=d(y_i,y_{i+1})$ (indices $\mod 3$) there is a canonical map
$f_\Delta:\Delta\to T(l_1,l_2,l_3)$.
One says that $X$ has {\em{$\eta$-thin triangles}} if ${\rm{diam}} (f_\Delta^{-1}(t))\le \eta$
for all $\Delta$ and all $t\in T(l_1,l_2,l_3)$; see figure \ref{fig1}. The preimage of the central 
vertex $c\in T(l_1,l_2,l_3)$ is the discrete analogue of the {\em{inscribed circle}}
of $\Delta$. 


\begin{figure} [h]\label{fig1}
\begin{tikzpicture}[scale=0.9][
  line cap=round,
  line join=round,
  >=Stealth,
  thick
]

\begin{scope}[shift={(-4.2,0)}]

  \coordinate (X)  at (-2.3,0.0);
  \coordinate (Y)  at ( 0.0,3.4);
  \coordinate (Z)  at ( 2.3,0.0);

  \coordinate (Mz) at (-0.85,1.55);
  \coordinate (Mx) at ( 0.85,1.55);
  \coordinate (My) at ( 0.00,0.5);

\coordinate (O) at (0.0,1.3);
\draw[line width=1.2pt] (O) circle [radius=0.85];


\begin{scope}
  \clip (X) .. controls (-0.75,1.55) .. (Y)
        .. controls (0.75,1.55) .. (Z)
        .. controls (0.00,0.65) .. cycle;

  \begin{scope}
    \clip (Mz) -- (Mx) -- (Y) -- cycle;
    \foreach \yy in {1.7,1.89,...,3.3} {
      \draw[line width=0.25pt] (-4,\yy) -- (4,\yy);
    }
  \end{scope}
  
\begin{scope}
  \clip (X) .. controls (-0.75,1.55) .. (Y)
        .. controls (0.75,1.55) .. (Z)
        .. controls (0.00,0.65) .. cycle;

  \begin{scope}
    \clip (X) -- (Mz) -- (My) -- cycle;

   \def\m{-1.2352941}

    \foreach \c in {-2.2,-1.9,...,0.4}  {
      \draw[line width=0.25pt] (-4,{\m*(-4)+\c}) -- (4,{\m*(4)+\c});
    }
  \end{scope}
\end{scope}

\begin{scope}
  \clip (X) .. controls (-0.75,1.55) .. (Y)
        .. controls (0.75,1.55) .. (Z)
        .. controls (0.00,0.65) .. cycle;

  \begin{scope}
    \clip (My) -- (Mx) -- (Z) -- cycle;

   \def\m{1.2352941}

    \foreach \c in {-2.2,-1.9,...,0.4} {
      \draw[line width=0.25pt] (-4,{\m*(-4)+\c}) -- (4,{\m*(4)+\c});
    }
  \end{scope}
\end{scope} 

\end{scope} 
  
  \draw (X) .. controls (-0.75,1.55) .. (Y);
  \draw (Y) .. controls (0.75,1.55) .. (Z);
  \draw (X) .. controls (0.00,0.65) .. (Z);

  \draw (Mz) -- (Mx);
  \draw (Mz) -- (My);
  \draw (My) -- (Mx);

  \fill (X)  circle (2pt) node[below left] {$y_1$};
  \fill (Y)  circle (2pt) node[above] {$y_3$};
  \fill (Z)  circle (2pt) node[below right] {$y_2$};

  \fill (Mz) circle (2pt) node[left] { $c_2$};
  \fill (Mx) circle (2pt) node[right] { $c_1$ };
  \fill  (0.00,0.45) circle (2pt) node[below] {  $c_3$};

  \node at (-0.05,1.05) {$\ \le \eta$};

\end{scope}

\draw[-{Stealth[length=3mm]}] (-0.65,1.6) -- (0.65,1.6);
\node at (0,1.95) { };

\begin{scope}[shift={(4.2,0)}]

  \coordinate (Ytop) at (0,3.0);
  \coordinate (M)    at (0,1.05);
  \coordinate (Xr)   at (-2.0,0.0);
  \coordinate (Zr)   at ( 2.0,0.0);

  \draw (Ytop) -- (M);
  \draw (M) -- (Xr);
  \draw (M) -- (Zr);

  \fill (Ytop) circle (2pt) node[above] {$v_3$};
  \fill (M)    circle (2pt) node[above right] {$\ c$};
  \fill (Xr)   circle (2pt) node[below left] {$v_1$};
  \fill (Zr)   circle (2pt) node[below right] {$v_2$}; 

\end{scope}

\end{tikzpicture}
\caption{Thin triangles} 
\end{figure} 


\begin{lemma}\label{l:slim}
A  group $G$ with finite generating set $X$ is hyperbolic if and only if there is a constant $\eta>0$ such that
all triangles in the Cayley graph $\cay(G,X)$ are $\eta$-thin. 
\end{lemma}

\section{Avoiding Proper Powers}\label{s:no-roots}

In this section we prove Proposition \ref{t:no-power} and deduce Proposition \ref{p:no-roots}. 
The proof of Proposition \ref{t:no-power}  is surprisingly long and challenging.  I regret that I have been unable to find a
more elementary proof of Proposition \ref{p:no-roots}.

\firstletterprop*

\subsection{Set-up and strategy}

We assume that $G$ is $\delta$-hyperbolic with respect to a fixed finite generating set $X$
and that geodesic triangles in $\cay(G,X)$ are $\eta$-thin  in the sense of section \ref{s:slim}. 
For each $g\in G$ we select a geodesic $\sigma_g\in F(X)$ representing it.  To avoid a clutter
of notation, we write $d$ instead of $d_G$.

We assume that constants $N$ and $K$ have been chosen with $N>\!\!> k >\!\!> 0$,
where $k=d(1,a^K)$.  The values of $N$ and $k$ will be specified more precisely at the end of the proof; they will be chosen
so that $k$ and $N-k$ dwarf  various combinations of the other constants that arise in the proof.   Suppose
$$d(1,h) >  N \hbox{    $>\!\!>$    } k=d(1,a^K).$$ 

Consider the Cayley graph $\Q_a:=\cay(G/\<\!\<a\>\!\>, X)$.
The quotient map $\cay(G,X)\to \Q_a$ preserves the lengths of
paths and hence does not increase distances.  If two paths in  $\Q_a$ are labelled by the
same word $w$ in the generators, then these two paths have the same length in $\Q_a$
and the same diameter, which we denote by $$\diam_{\Q}(w).$$ We also write $$\diam_{\Q}(S)$$ for the 
diameter in $\Q_a$ of  subsets  $S\subset \cay(G,X)$.

The constant $\xi := \diam_\Q(\sigma_a)$ will play a role in the proof, as will the
constants $E$ and $\mu$ from Lemmas \ref{l:constants}  and
the constant $C_0$ obtained by appealing to Lemma \ref{l:constants} 
with the constants $\l_0, \e_0$ from  Lemma \ref{l:local-geod} in place of $\l,\e$. 

We want to derive a contradiction from the assumption that there  exists $g\in G$ with $g^p=ha^K$
and $p\ge 2$. The proof will divide into various cases, each giving rise to a different geometric configuration.
In each case the strategy for obtaining a contradiction is the same: we want to identify two long arcs that are close in $\cay(G,X)$ but are such that the image of one has small diameter in $\Q_a$ while the image of the other is large.

The case where $p$ is odd is more complicated than the case where $p$ is even.   We begin with the
even case, which immediately reduces to the case $p=2$.

\subsection{Excluding the possibility that $g^2=ha^K$}\label{s:even}

In order to analyse this possibility, we consider the configuration in figure \ref{fig2}. This portrays
two geodesic triangles in the Cayley
graph of $G$ with their inscribed circles drawn;   these triangles are $\Delta(1,h,ha^K)$
and   $\Delta(1,g,g^2)$. The  vertices are labelled by group elements.  The top arc marked $a^K$ connects
$h$ to $ha^K$ and is not assumed to be a  geodesic, rather it is the path in $\cay(G,X)$ beginning at $h$
that is labelled $\sigma_a^K$; we write $a^K$ rather than $\sigma_a^K$ to limit the explosion of notation. 
For some $\l\ge 1$ and $\e\ge 0$, this arc is a $(\l, \e)$-quasigeodesic, so it is a Hausdorff
distance at most $E$ from the geodesic
$[h, ha^K]$ that is drawn horizontally, where  $E$ is the constant of Lemma \ref{l:constants}(1). 
Increasing $E$ by $|\sigma_{a}|/2$ if necessary, we can assume that every point of $[h, ha^K]$ is within
a distance $E$ of a vertex $ha^i$ on the top arc.
 

\begin{figure}\label{fig2}
\begin{tikzpicture}[
    line cap=round,
    line join=round,
    >=Stealth,
    scale=0.95
]

\coordinate (A) at (0,0);      
\coordinate (B) at (0,6.0);    
\coordinate (C) at (5.7,6.0);  
\coordinate (D) at (1.35,3.55);  
\coordinate (E) at (3.6,2.6);  
\coordinate (F) at (1.8,4.4);  
\coordinate (G) at (1.0,5.0);  
\coordinate (M) at (2.5,2.5);  
\coordinate (N) at  (3.2,3.4); 
\coordinate (W) at  (3.5,4.3); 

\draw[thick] (A) -- (B);
\draw[thick] (B) -- (C);

\draw[thick] (G) circle[x radius=0.99, y radius=0.99];
\draw[thick] (B) .. controls (3.0,6.4)  .. (C);

\draw[thick] (2.25,3.3) ellipse[x radius=0.95, y radius=0.82];

\draw[thick] (A) .. controls (1.0,3.0) .. (F);
\draw[thick] (A) .. controls (1.8,2.6) and (M) .. (E);
\draw[thick] (F) .. controls (3.0,5.5) .. (C);
\draw[thick] (E) .. controls (3.15,3.4) .. (W);
\draw[thick] (W) .. controls (4.3,5.5) .. (C);

\draw[thick, -{Stealth[length=3mm]}] (3.9,2.6) -- (6.0,5.7);
\draw[thick, -{Stealth[length=3mm]}] (0.3,0.0) -- (3.6,2.4);

\fill (A) circle (1.6pt) node[below left] {$1$};
\fill (B) circle (1.6pt) node[left] {$h$};
\fill (C) circle (1.6pt) node[above right] {$ha^K=g^2$};
\fill (D) circle (1.4pt) node[below right] {$m$};
\fill (E) circle (1.6pt) node[below right] {$\ g$};
\fill (F) circle (1.4pt) node[above left] {$c_3$};
\fill (M) circle (1.4pt) node[above] {$m_1$};
\fill (N) circle (1.4pt) node[left] {$m_2$};
\fill (W) circle (1.4pt) node[right] {$\ c_3'$};

\fill (0,5.0) circle (1.2pt) node[right] {$c_1$};
\fill (1.0,6.0) circle (1.2pt) node[below] {$c_2$};

\node at (2.9,6.55) {$a^K$};
\node at (3.0,6.0) {$\raisebox{-0.2ex}{\scalebox{1.2}{$\rightarrow$}}$};
\node at (3.0,5.7) {$\sigma_{a^K}$};
\node at (5.3,4.0) {$\sigma_g$};
\node at (2.2,1.0) {$\sigma_g$};
\node  at (0.0, 3.0) {$\raisebox{0.2ex}{\scalebox{1.2}{$\uparrow$}}$ };
\node at (0.0,3.0)[left] {$\sigma_h$};

\draw[thick,<->] (-1.0,0.0) -- (-1.0,6.0);
\node at (-1.45,3.1) {$\ge N$};

\end{tikzpicture}

\caption{$g^2=ha^K$} 
\end{figure} 


The side $[1,h]\subset \Delta(1,h,a^K)$ is labelled by the geodesic word $\sigma_h$ and
the sides $[1,g]$ and $[g,g^2]$ of $\Delta(1,g,g^2)$ are both labelled by the  geodesic
word $\sigma_g$, while $[1,g^2]=[1,ha^K]$ is labelled by the geodesic word $\sigma_{g^2}$. 
The inscribed circle of the  isosceles triangle $\Delta(1,g,g^2)$ intersects $[1,g^2]$ at the midpoint $m$.

Because $d(1,h)$ is the distance between  $1$ and the image of $h$ in $\Q_a$, we have
 $d(1,h)\le d(1, ha^i)$ for all $i\in\Z$.   
The geodesic  $[h,g^2]$ is $E$-close to the set of vertices $ha^i$, 
so the path  $[1, c_1, c_2]$ has length at least $d(1,h)-E$ and
\begin{equation}\label{e1}
d(h,c_2) = d(h,c_1) \le d(c_1, c_2) + E \le \eta + E.
\end{equation} 
From this we obtain a lower bound on $d(c_3, g^2)$,
\begin{equation}\label{e3}
d(c_3, g^2) = d(c_2, g^2) = d(h,g^2) - d(h,c_2) \ge k - (\eta + E).
\end{equation}
For the midpoint $m$ we have,  from the triangle inequality,  
\begin{equation}\label{e4}
d( m, g^2) = \frac{1}{2} d(1, g^2)\ge\frac{1}{2} ( d(1,h) - d(h, g^2)) \ge \frac{1}{2} (N-k) >\!\!> k.
\end{equation}
Comparing this with $d(c_3, g^2) = d(c_2, g^2) = d(h,g^2) - d(h,c_2)\le k$, 
we see that $m$ lies between $1$ and $c_3$, as drawn. It follows that the point
$c_3'\in [g,g^2]$ a distance $d(c_3,g^2)$ from $g^2$ lies on the arc $[m_2,g^2]$. Let $\vi$ be the
suffix of $\sigma_g$ that labels the arc $\tau_0:=[c_3', g^2]$.  From (\ref{e3}) we have
\begin{equation}\label{e5}
|\vi| = |\tau_0| = d(c_3,g^2)  \ge k -(\eta + E).
\end{equation}
This arc $\tau_0$ lies in the $\eta$-neighbourhood of $[c_3, g^2]$,  hence in the $2\eta$-neighbourhood of $[c_2, g^2]$
and in the  $(2\eta+E)$-neighbourhood of the arc labelled $a^K$ at the top of the picture. The image of this last
arc in $\Q_a$ has diameter $\xi$,  therefore
\begin{equation}\label{e6}
\diam_{\Q}(\vi) \le 2\eta + E + \xi.
\end{equation}

We now focus on a neighbourhood of the vertex $g$ in figure \ref{fig2} as well as the arc $[1,g]$ and the initial segment of $[1,h]$ that has length $d(1,g)$; see figure \ref{fig3}


\begin{figure}\label{fig3}

\begin{tikzpicture}[
    line cap=round,
    line join=round,
    >=Stealth,
    scale=1.0
]

\coordinate (P) at (0,0);        
\coordinate (U) at (0,4.5);      
\coordinate (M) at (4.55,1.85);  
\coordinate (N) at (5.0,1.52);  
\coordinate (O) at (5.9 , 2.0);  
\coordinate (G) at (6.2,0.55);   
\coordinate (Mtwo) at (5.87,2.75); 
\coordinate (C) at (4.85,2.85);  

\draw[thick] (P) -- (U);
\draw[thick,postaction={decorate},decoration={
    markings,
    mark=at position 1.0  with {\arrow{Stealth[length=3mm,width=2mm]}}
}] (P) -- (0,1.85);
\draw[line width=2.5pt] (P) -- (0, 1.7);

\node[left] at (-0.25,0.85) {$\tau_3$};
\draw[thick,  dotted] (0,4.5) -- (0,4.95); 
\node[below left] at (P) {$1$};

\draw[thick,postaction={decorate},decoration={
    markings,
    mark=at position 0.35 with {\arrow{Stealth[length=3mm,width=2mm]}}
}] (P) .. controls (1.15,1.85) and (2.75,2.65) .. (M);
\draw[line width=2.5pt] (P) ..controls (0.5,0.75).. (1.05,1.3);
\node at (0.85,0.55) {$\tau_2'$};

\draw[thick] (C) circle [radius=1.0];
\fill (M) circle (1.8pt) node[below] {$m_1$};
\fill (Mtwo) circle (1.8pt) node[right] {$m_2$}; 

\draw[thick](M) -- (G);
\draw[line width=2.5pt] (N) -- (G);
\node[below] at (5.2,1.0) {$\tau_1$};

\draw[thick,postaction={decorate},decoration={
    markings,
    mark=at position 0.75 with {\arrow{Stealth[length=3mm,width=2mm]}}
}] (G) .. controls (6.10,0.95) and (5.87,1.65) .. (Mtwo);
\draw[line width=2.5pt] (G) ..controls (6.1, 0.9).. (O);
\node[right] at (6.05,1.25) {$\tau_2$};

\draw[thick] (Mtwo) -- (5.87,4.35);
\draw[thick,  dotted] (5.87,4.35) -- (5.87,4.9); 

\fill (G) circle (1.8pt) node[below right] {$g$};
\fill (P) circle (1.8pt);

\end{tikzpicture}
        \caption{The arcs $\tau_i$}
\end{figure}


\smallskip
\noindent{\em Case 1: Assume $d(m_1,g) \ge |\vi|/2$}}
\smallskip

In this case, a terminal arc of $[m_1,g]$ is labelled by the suffix of $\vi$ that has length  $|\vi|/2$. This arc, which
is labelled $\tau_1$ in figure \ref{fig3}, is $\eta$-close to an initial arc $\tau_2 \subset [g,g^2]$ of the same length.
Because $[1,g]$ and $[g,g^2]$ are both labelled by the same geodesic word $\sigma_g$,  there is an initial arc
of $[1,g]$ with the same label as $\tau_2$; in figure \ref{fig3}  this arc is called $\tau_2'$.  As
$$
|\tau_2'| = \frac{1}{2} |\vi | \le \frac{1}{2} k <\!\!< N,
$$
the arc $\tau_2'$ is $2\eta$-close to an initial arc of $[1,h]$ of the same length; this arc is called $\tau_3$ in figure \ref{fig3}.

Now the desired contradiction emerges, because on  the one hand, using (\ref{e5}), 
\begin{equation}\label{e7}
\diam_{\Q}(\tau_3) = |\tau_3|= |\tau_2| = |\tau_1|  \ge \frac{1}{2}|\vi| \ge  \frac{1}{2} (k - (\eta +E)),
\end{equation}
while  on the other hand, since $\tau_1$ is labelled by a suffix of $\vi$, from (\ref{e6}) we have
\begin{equation}\label{e8}
\diam_{\Q}(\tau_1) \le \diam_{\Q}(\vi) \le 2\eta + E + \xi.
\end{equation}
Moreover,    $\tau_1$  is $\eta$-close to   $\tau_2$,  which shares a label with $\tau_2'$, which 
is $2\eta$-close to  $\tau_3$, so
\begin{equation}\label{e9}
\diam_{\Q}(\tau_3) \le \diam_{\Q}(\tau_2') + 2\eta = \diam_{\Q}(\tau_2) + 2\eta \le \diam_{\Q}(\tau_1) + 3\eta.
\end{equation}
Together,  (\ref{e7}), (\ref{e8}) and (\ref{e9}),  imply  
\begin{equation}\label{e10} 
\frac{1}{2} (k - (\eta +E)) \le 5 \eta + E + \xi, 
\end{equation}
which is nonsense if $k$ is large.

\smallskip
\noindent{\em Case 2: Assume $d(m_1,g) \le |\vi|/2$}
\smallskip

In this case the terminal segment of $[1,g]$ labelled $\vi$ begins with an arc of length at least $|\vi|/2$ on $[1,m_1]$.
This arc is $2\eta$-close to an arc of the same length on $[1,h]$, which has diameter 
at least $|\vi|/2$ in $\Q_a$, so we
reach a contradiction  as before. 

\subsection{Excluding the possibility that $g^p=ha^K$ with $p$ odd}

In this case we will reach a contradiction by finding a long subword of $\sigma_{g^p}$ that labels
two paths in $\cay(G,X)$, one of which has small diameter in $\Q_a$ and one of which has large diameter.
The reader may find it helpful at each stage of the proof to reflect on what happens in the case where $G$ is a free group.

Given $g\in G$,  we  consider those words $\theta$ in the generators of $G$ 
that are shortest  among all words representing elements
in the conjugacy class of $g$; suppose $\theta=x_\theta^{-1}g x_\theta$ in $G$.
Among these, we fix a particular $\theta$ so that $d(1,x_\theta)$ is minimal, 
and we fix a shortest word $x$ representing $x_\theta$.  If $G$ were a free group,
$x\theta^px^{-1}$ would be the geodesic representative of $g^p$ for all $p>0$, but in the general case a
non-trivial argument is needed to bound $|x|$ in terms of $d(1,g^p)$.

There is a  line in the Cayley graph labelled $\theta^*$ through the vertex $x_\theta$;  this line is a $|\theta|$-local geodesic but might not
be a geodesic.  A shortest path from the vertex $1$ to this line is labelled
$x$. The following argument shows
that $x$ is also the label on a shortest path from the vertex $g^p$ to this line for all $p>0$.

Let $x_1$ be any word that conjugates $g^p$ to an element
represented by a cyclic permutation of $\theta^p$, where $p> 0$;
say $x_1^{-1} g^p x_1 = \theta_0^{-1}\theta^p \theta_0$, with $\theta_0$ a prefix of $\theta$.
Then, because roots are unique in torsion-free hyperbolic groups, 
$x_1$ conjugates $g$ to $\theta_0^{-1}\theta \theta_0$, which being a cyclic permutation of
$\theta$ has reduced length $|\theta|$. Therefore, 
\begin{equation}\label{e:least}
|x_1|\ge |x|=d(1,x_\theta),
\end{equation}
as $|x|$ was chosen to be minimal.

We now fix $x$ (which is a geodesic word),  define $g_0\in G$ by
$$
\boxed{g_0=x^{-1}gx}
$$
 and work with the geodesic representative $\theta$ for $g_0$.  
Lemma \ref{l:local-geod}(1) provides constants $\l_0,\e_0$ such that  the lines in $\cay(G,X)$ labelled $\theta^*$ are
$(\l_0,\e_0)$-quasi-geodesic for all $\theta$ under consideration. We fix a constant $C_0>0$
such that every $(\l_0,\e_0)$-quasigeodesic  quadrilateral in $H$
is $C_0$-slim (Lemma \ref{l:constants}(2)).

Recalling our notation that $\sigma_{g^p}$ is a 
 geodesic word  for $g^p$, we  consider the quadrilateral in the Cayley graph of $G$  with
vertices $y_1, y_2, y_3, y_4$ and sides (read
in cyclic order) labelled $x^{-1},\sigma_{g^p}, x, \theta^{-p}$ (see figure \ref{fig4}).
An important point to observe is that there is no path in the Cayley graph of length less than $|x|$ 
that connects either  $y_2$ or $y_3$ to the  fourth side (the arc labelled $\theta^p$), for if there were then the label
$x_1$ on this path would contradict the minimality of $|x|$, as in the argument leading to (\ref{e:least}).


\begin{figure}[ht]
\begin{tikzpicture}[line cap=round, line join=round, >=Stealth]

\coordinate (y1) at (0,3);
\coordinate (y2) at (0,0);
\coordinate (y3) at (5,0);
\coordinate (y4) at (5,3);

\coordinate (z)  at (0,1.55);
\coordinate (y0) at (5,1.55);

\draw[thick] (y1) .. controls (1.7,3.28) and (3.3,3.28) .. (y4);
\draw[thick] (y2) -- (y3);

\draw[thick] (y2) -- (y1);
\draw[thick] (y3) -- (y4);

\draw[thick] (z) .. controls (2.2,1.8) and (3.1,1.8) .. (y0);

\fill (y1) circle (2pt) node[left] {$y_1$};
\fill (y2) circle (2pt) node[below] {$y_2$};
\fill (y3) circle (2pt) node[below] {$y_3$};
\fill (y4) circle (2pt) node[right] {$y_4$};
\fill (z)  circle (2pt) node[left] {$z$};
\fill (y0) circle (2pt) node[right] {$y_0$}; 

\node at (2.6,3.6) {$\theta^p$};

\draw[-{Stealth[length=3mm]}] (1.2,3.35) -- (4.0,3.35);

\draw[-{Stealth[length=3mm]}] (1.1,-0.6) -- (3.9,-0.6);
\node at (2.5,-0.93) {$\sigma_{g^p}$};

\draw[-{Stealth[length=3mm]}] (-1,0.4) -- (-1,2.7);
\node[left] at (-1,1.55) {$x$};

\draw[-{Stealth[length=3mm]}] (6,0.4) -- (6,2.7);
\node[right] at (6,1.55) {$x$};

\draw[thick] (y2) ++(0.40,0) arc (0:90:0.40); 
\draw[thick] (y2) ++(0.70,0) arc (0:90:0.70); 
\draw[thick] (y2) ++(1.00,0) arc (0:90:1.00); 
\draw[thick] (y2) ++(1.30,0) arc (0:90:1.30); 
\draw[thick] (y2) ++(1.55,0) arc (0:90:1.55); 

\node at (-0.2,0.72) {$\ell$};

\end{tikzpicture}
        \caption{When $|x|$ is significant}\label{fig4}
\end{figure} 


\smallskip
\noindent{\em Case 1: When $|x|$ is significant ($|x|\ge k/100$ will suffice).}
\smallskip

The key point to prove is that if $x$ is long, then a long initial arc of $[y_2,y_1]$ is in the $C_0$-neighbourhood of $[y_2,y_3]$.
To this end, consider the  first vertex $z\in [y_2,y_1]$ that is not in this neighbourhood and let $\ell = d(y_2,z)$. We 
shall prove that
\begin{equation}\label{e:ell}
\ell \ge \min\ \biggl\{\frac{1}{2}d(1,g^p) -C_0,\  |x|/2\biggr\}.
\end{equation} 
If $d(y_2, z) \ge d(y_1,z)$ then we are done, so suppose $d(y_1,z) \ge |x|/2 > C_0$. This  supposition 
precludes $z$ from being within
$C_0$ of  $[y_1,y_4]$,  by the minimality of $|x|$, so
 $z$ is a distance at most $C_0$ from   a  point $y_0\in [y_3,y_4]$. 
Note that $d(y_4,y_0)\ge |x| - \ell - C_0$, for if not then
$d(y_2,y_4) \le d(y_2, z) + d(z, y_0) + d(y_0, y_4) < |x|$,
contradicting the minimality of $|x|$.   
It follows that $d(y_0, y_3)\le \ell  + C_0$, and measuring the path $(y_2, z, y_0, y_3)$ we conclude that
$$d(1,g^p) = |\sigma_{g^p}| = d(y_2, y_3) \le 2(\ell + C_0),$$ 
which completes the proof of (\ref{e:ell}).

An entirely symmetric argument provides the same lower bound on the length of an initial segment of $[y_3,y_4]$
that lies in the $C_0$-neighbourhood of a terminal segment of $\sigma_{g^p}$. 

In the setting that  concerns us (figure \ref{fig5}),  $d(1,g^p)= d(1, ha^K)>\!\!> k$, so   
if $|x|\ge k/100$ then $x$ has a prefix $x'$ of length at least $k/200$ that labels two geodesic arcs  in the Cayley graph, 
the first of which is $C_0$-close to an initial segment of $\sigma_{g^p}$ and the second of which is
$C_0$-close to a terminal segment of $\sigma_{g^p}$.  In figure \ref{fig5},  this puts the first arc $(C_0+\eta)$-close
to a segment of the same length on $[1,h]$ and the second arc $(C_0+\eta)$-close to a segment of $[h, ha^K]$
and hence $(C_0+\eta+E)$-close to a segment of the arc labelled $a^K$.


\begin{figure}[ht]
\begin{tikzpicture}[
    line cap=round,
    line join=round,
    >=Stealth,
    scale=0.95
]

\coordinate (A) at (0,0);      
\coordinate (B) at (0,6.0);    
\coordinate (C) at (5.7,6.0);  
\coordinate (F) at (1.8,4.4);  
\coordinate (G) at (1.0,5.0);  
\coordinate (X) at (1.8,0.7);  
\coordinate (Xm) at (0.9,0.35); 
\coordinate (Y) at (5.7,4.2);  
\coordinate (Ym) at (5.7,5.1);  
\coordinate (Z) at (3.0,3.1);  

\coordinate (M) at (1.4,2.6);  
\coordinate (N) at (3.6,5.2);  

\draw[thick] (A) -- (B);
\draw[thick] (B) -- (C);
\draw[thick, ->] (A) -- (Xm); 
\draw[thick] (Xm) -- (X); 
\draw[thick, ->] (C) -- (Ym);
\draw[thick] (Ym) -- (Y);
  
\draw[thick] (G) circle[x radius=0.99, y radius=0.99]; 
\draw[thick] (B) .. controls (3.0,6.4)  .. (C);

\draw[thick] (A) .. controls (1.0,3.0) .. (F); 
\draw[thick] (F) .. controls (3.0,5.5) .. (C); 
\draw[thick, ->] (X) .. controls (2.2,2.1) ..  (Z);
\draw[thick] (Z) .. controls (4.2,4.0) .. (Y); 
 
\draw[thick, ->] (M) .. controls (1.8,3.9) and (2.6,4.6) ..  (N);

\fill (A) circle (1.6pt) node[below left] {$1$};
\fill (B) circle (1.6pt) node[left] {$h$};
\fill (C) circle (1.6pt) node[above right] {$ha^K=g^p$}; 
\fill (Y) circle (1.6pt) node[below right] {$xg_0^p$}; 
\fill (X) circle (1.6pt); 
\fill (F) circle (1.4pt) node[above left] {$c_3$};  
\fill (Xm) circle (0.1pt) node[below right] {$x$};  
\fill (Ym) circle (0.1pt) node[right] {$x$};  

\fill (0,5.0) circle (1.2pt); 
\fill (1.0,6.0) circle (1.2pt); 

\node at (2.9,6.6) {$a^K$}; 
\node at (3.6,3.0) {$\theta^p$}; 
\node at (2.5,3.9) {$\ \sigma_{g^p}$};

\end{tikzpicture}
        \caption{$g^p=ha^K$, general case}\label{fig5}
\end{figure}


These proximities give  contradictory bounds when $k$ is large:
\begin{equation}\label{e:small-x}
\diam_{\Q}(x') \ge |x'| - (C_0+\eta) \ge (k/200) - (C_0+\eta) 
\end{equation} 
versus
$$ 
\diam_{\Q}(x') \le C_0+\eta+E + \xi.
$$
Thus we have proved that if $|x|\ge k/100$ and $k$ is large, then $g = xg_0x^{-1}$ cannot be a  proper root of $ha^K$. 

\smallskip

There are two cases left to consider: either $d(1,g)$ is small (and hence $|x|$ is small), or $d(1,g)$ is large but $|x|$ is small.

\smallskip
\noindent{\em Case 2: When $d(1,g)$ is small}
\smallskip

This case is portrayed in figure \ref{fig6}.   We write $p=2q+1$, so   
$g^{2q+1}=ha^K$.  Suppose  $d(1,g)= D <k/10$, say.
Consider the juxtaposition of geodesic triangles $\Delta(h, g^{2q}, ha^K),\  \Delta(1, h,g^{2q})$
and $\Delta(1,g^q, g^{2q})$, with the latter having the sides $[1,g^q]$ and $[g^q, g^{2q}]$ both labelled $\sigma_{g^q}$.
\begin{figure}[ht]
\begin{tikzpicture}[
    line cap=round,
    line join=round,
    >=Stealth,
    scale=0.95
]

\coordinate (A) at (0,0);      
\coordinate (B) at (0,6.0);    
\coordinate (C) at (5.7,6.0);  
\coordinate (D) at (5.7, 5.0);  
\coordinate (V) at (4.5, 4.3);  

\coordinate (a1) at (4.45, 5.2);  
\coordinate (a2) at (4.70, 5.16);
\coordinate (a3) at (4.95, 5.12);
\coordinate (a4) at (5.20, 5.08);
\coordinate (a5) at (5.45, 5.04);

\coordinate (v2) at (4.74, 4.44); 
\coordinate (v3) at (4.98, 4.58);
\coordinate (v4) at (5.22, 4.72);
\coordinate (v5) at (5.46, 4.86);

\coordinate (E) at (3.2, 2.5);  
\coordinate (F) at (1.8,4.4);  
 
\coordinate (X) at (1.8,0.7);  
\coordinate (Xm) at (0.9,0.35); 
\coordinate (Y) at (5.7,4.2);  
\coordinate (Ym) at (5.7,5.1);  
\coordinate (Z) at (3.0,3.1);  

\coordinate (M) at (1.4,2.6);  
\coordinate (N) at (3.6,5.2);  

\draw[thick] (A) -- (B);
\draw[thick] (B) -- (C);
\draw[thick] (B) -- (D);
\draw[thick] (D) -- (C);  
 
  
\draw[line width=2.5pt] (V)--(D);


\draw (a1) -- (V);
\draw (a2) -- (v2);
\draw (a3) -- (v3);
\draw (a4) -- (v4);
\draw (a5) -- (v5); 


\draw[thick] (B) .. controls (3.0,6.4)  .. (C); 
\draw[thick, ->] (A) .. controls (1.0,3.8) and (1.2,4.3) .. (F); 
\draw[thick] (A) .. controls (1.5,3.0) .. (E);  
\draw[thick] (F) .. controls (3.8,4.8) .. (D);  
\draw[thick] (E) .. controls (3.8,3.8) .. (V);   

\fill (A) circle (1.6pt) node[below left] {$1$};
\fill (B) circle (1.6pt) node[left] {$h$};
\fill (C) circle (1.6pt) node[above right] {$ha^K=g^{2q+1}$}; 
\fill (D) circle (1.6pt) node[right] {$\ g^{2q}$};   
\fill (E) circle (1.6pt) node[below right] {$g^{q}$};

\node at (2.9,6.6) {$a^K$};  
\node at (1.45,4.7) {$\ \sigma_{g^{2q}}$}; 
\node at (v3)[below right] {$v_\infty$}; 

\end{tikzpicture}
        \caption{When $g$ is small}\label{fig6}
\end{figure} 
As $D $ is significantly smaller than $k$, 
 there is a terminal arc of length $k/2$ on $[g^q, g^{2q}]$ that is $2\eta$-close to a terminal
arc of the same length on $[h, g^{2q}]$;  let $ v_\infty$ be the label on this arc of $[g^q, g^{2q}]$.  
The thinness of $\Delta(h, g^{2q}, ha^K)$ tells us that $[h, g^{2q}]$ is contained in the $(\eta+D)$-neighbourhood of 
$[h, ha^K]$ and hence the $(\eta+D+E)$-neighbourhood of the arc labelled $a^K$.   Therefore
\begin{equation}
\diam_{\Q}(\vi) \le 3\eta+D+E + \xi < 3\eta + \frac{1}{10}k+E +\xi.
\end{equation}
With $\vi$ in hand, we concentrate on $[1,h]\cup \Delta(1,g^q, g^{2q})$ and reach a contradiction by
identify arcs 
$$
\tau_1\subset [1,g^q],\ \tau_2\subset [g^q,g^{2q}], \ \tau_2'\subset [1,g^q] \hbox{   and  }\tau_3\subset [1,h]
$$
exactly as we did in Section \ref{s:even} (with constants that have been changed by at most $D$).

\smallskip
\noindent{\em Case 3: When $d(1,g)$  is large but $|x|$ is small.}
\smallskip

We are working with  the decomposition  $g=x^{-1}g_0x$, where  $x$ is a geodesic and $g_0$ is represented by a shortest word representing any conjugate of $g$; let $\theta$ be this word.  
To cover the remaining cases, it suffices to assume that   
 $|x| \le \iota = k/100$ and  $|\theta|\ge (k/10) - 2\iota$, which we approximate by $k/20$ for convenience.  
In this case, we consider the juxtaposition of the geodesic triangle $\Delta(1,h,ha^K)$ (with
$[1,  ha^K]$ labelled $\sigma_{g^p}$) and the quadrilateral  that has sides labelled $\sigma_{g^p},  \ \theta^p$ and, on the
remaining two sides, $x$.   From
Lemmas \ref{l:constants} and \ref{l:local-geod},
we know that this quadrilateral is $C_0$-thin.   
Figure \ref{fig4} portrays this situation.  Consider the terminal arc of length $k/50$ on the side of the
quadrilateral labelled $ \theta^p$.  This arc (which is labelled by a suffix $\zeta$ of $\theta$, since  $|\theta| \ge k/20$)
is contained in the  $(\iota + C_0)$-neighbourhood of a terminal arc of $[c_3,ha^K]\subset [1,ha^K]$ and hence
lies in the $(\iota + C_0+\eta + E)$-neighbourhood of the arc labelled $a^K$ joining $h$ to $ha^K$. Therefore
\begin{equation}\label{e:almost}
\diam_{\Q}(\zeta) \le \iota + C_0+\eta + E + \xi = \frac{1}{100}k + C_0+\eta + E + \xi.
\end{equation}
Now consider the segment labelled $\theta$ at the beginning
on the  $\theta^p$-side of the quadrilateral,   remembering that $\theta$ is a geodesic word.
Let $\omega$ be the terminal part of
this segment that is labelled $\zeta$ (see figure \ref{fig7}).

\begin{figure}[ht]
\begin{center}
\begin{tikzpicture}[
    line cap=round,
    line join=round,
    >=Stealth,
    scale=0.95
]

\coordinate (A) at (0,0);      
\coordinate (B) at (0,4.0);    
\coordinate (C) at (2.0,4.0);  
\coordinate (D) at (3.3,2.6);  
\coordinate (E) at  (4.0, 4.0);  
\coordinate (e1) at  (3.5, 3.0);  
\coordinate (X) at (2.0,0);  

\coordinate (b1) at (0.0, 1.0);  
\coordinate (b2) at (0.0, 2.0); 

\coordinate (c1) at (0.5, 1.0); 
\coordinate (c2) at (1.0, 2.0); 

\coordinate (d1) at (2.3, 0.6); 
\coordinate (d2) at (2.8, 1.6);

\draw[->]  (2.5, 0.6)-- (2.9, 1.4);

\draw[thick] (A) -- (b1);
\draw[thick, ->] (b2) -- (B);
\draw[thick] (A) -- (c1);
\draw[thick, ->] (c2) -- (C);  
\draw[thick] (X) -- (d1);
\draw[thick] (d2) -- (e1);  
\draw[thick] (A) -- (X);
 
  
\draw[line width=2.5pt] (b1)--(b2);
\draw[line width=2.5pt] (c1)--(c2);
\draw[line width=2.5pt] (d1)--(d2);
 

\draw[thick,  dotted] (E) -- (e1); 
 
\fill (A) circle (1.6pt) node[below left] {$1$}; 
\fill (e1) circle (1.6pt) node[right] {$x\theta^2$};     
\fill (d2) circle (1.6pt) node[right] {$x\theta$};      
\fill (X) circle (1.6pt) node[below right] {$x$};       
 
\node at (-0.3,1.5) {$\gamma$};  
\node at (1.2,1.5) {$\beta$};  
\node at (2.1,1.1) {$\omega$};  
\node at (2.9,0.9) {$\xi$};  

\node at (0.5,3.6) {$\sigma_h$};  
\node at (2.5,3.6) {$\sigma_{g^p}$};  

\end{tikzpicture}
\end{center}
        \caption{The final set of contradictory arcs}\label{fig7}
\end{figure} 
The thinness of the quadrilateral (Lemma \ref{l:local-geod}), tells us that  $\omega$
 is $C_0$-close to a geodesic segment of  length at least $|\omega| - 2C_0 \ge (k/50) - 2C_0$
on $[1,ha^K]$.   If we can argue that this
last segment, which we call $\beta$,
 is contained in $[1,c_3]$, then by the thinness of $\Delta(1,h,ha^K)$ it will be $\eta$-close to
a segment $\gamma$ of the same length on $[1,h]$,  and we know that $\diam_{\Q}(\gamma) = |\gamma|$.  As $\gamma$ lies in the
$C_0 + \eta$ neighbourhood of the arc labelled $\zeta$, we could then conclude
that 
\begin{equation} \label{e15}
\diam_{\Q}(\zeta)  \ge \diam_{\Q}(\gamma) - C_0 - \eta \ge (k/50) - 3C_0 -\eta,
\end{equation}
which contradicts (\ref{e:almost}) if $k$ is large enough; and this contradiction will finish the proof.

Thus it only remains to prove that $\beta$ is indeed contained in $[1,c_3]$ (and hence is $\eta$-close to $[1,h]$).
Since $\beta$ is $C_0$-close to $\omega$,  it suffices to 
prove that the latter is contained in a ball of radius less than $N-k-C_0$ about $1$.
And since $\omega$ lies at the end of an arc labelled $\theta$ that begins 
at $x$, it suffices to prove that 
\begin{equation}\label{e16}
|\theta| < N-k - \iota - C_0 =  N - \frac{101}{100} k -C_0.
\end{equation}

We are in the case where $|\theta| \ge k/20 > 12\delta$,  so from  Lemma \ref{l:local-geod}(2) we know that  
\begin{equation}
d(1, g_0^p) \ge \frac{p}{2} |\theta| - 2\delta.
\end{equation}
By the triangle inequality, 
\begin{equation}
N+k \ge d(1, ha^K) =  d(1,  x^{-1} g_0^px) \ge d(1, g_0^p)  - 2|x| .
\end{equation}
As $|x|\le \iota$,  we deduce
\begin{equation}\label{e19}
\frac{p}{2}  |\theta|  \le   N+k  +  2\iota + 2\delta.
\end{equation}
And since $p\ge 3$,  
\begin{equation}\label{e20}
 |\theta|  \le  \frac{2}{3}   ( N+k  +  2\iota + 2\delta\ ) =  \frac{2}{3} \big( N+\frac{51}{50}k  +  2\delta\big).
\end{equation}
We choose  constants $N>\!\!>k$ to ensure that this last quantity is   less than $N-k - \iota - C_0$,
so the estimate (\ref{e16}) is established and the proof is complete. 
\qed

\smallskip
\noindent{\bf Values of $N, K$ and $k$.} The reader who has successfully followed the above proof
will be content that a choice of $N>\!\!>k\!\!>0$ can be made to ensure that all of the desired contradictions are
reached.  For example,   $k$ has to be greater than  $11\eta +3E +2\xi$ to ensure that (\ref{e10}) is false,  it
has to be greater than $200(2C_0 + 2\eta + E + \xi)$ to obtain a contradiction from (\ref{e:small-x}),
and it has to be greater than $50(4C_0+2\eta+E+\xi)$ in order for (\ref{e15}) to contradict (\ref{e:almost}).
When we have settled on what a sufficiently large value of $k$ is, we choose $K$ to be the least
positive integer such that $d(1,a^K)>k$ and then we increase $k$ to get $d(1,a^K)=k$.  After this, $N$ has to be chosen
so that  the estimate in  (\ref{e20})  implies the one in (\ref{e16});  it is sufficient to let $N = 3k+3C_0+4\delta$.
(Psychologically speaking, it can be helpful to imagine that   $N/k$ is much bigger.)

\subsection{Proof of Proposition \ref{p:no-roots}}

\otherletterprop*

\begin{proof}   By composing $F(X)\to G$ with $p:G\to Q$ we can regard $X$ as a finite generating set 
for $Q$ as well as $G$, and $p$ extends to a length-preserving map $\cay(G,X)\to \cay(Q,X)$.  
If $a\in \ker (G\to Q)$, then $\cay(G,X)\to \cay(Q,X)$ factors through $\cay(G,X)\to \Q_a$,
so if  $w_0\in F(X)$ arises as the label on a geodesic path in $\cay(Q,X)$, then
the paths in $\Q_a$ labelled $w_0$ are also geodesics.  

We fix a non-trivial element $a\in \ker (G\to Q)$,  choose a word $\tilde{a}\in F(X)$ representing $a$,
 and fix constants $N$ and $K$ satisfying the statement of Proposition \ref{t:no-power}. 
 We define $\EE\subset F(X)$ to be the set of  words labelling geodesics of length less than $N$ in $\cay(Q,X)$.  
The algorithm that we seek  proceeds as follows: 
given  $w\in F(X)$,  it first uses the solution to the word problem in $Q$ to find a word $w_0\in F(X)$ of
minimal length such that $p(w)=p(w_0)$ in $Q$.  We have just observed that the paths in $\Q_a$ labelled
$w_0$ are geodesics,  so the element $h\in G$ defined by $w_0$ satisfies the hypotheses of  Proposition \ref{t:no-power}
unless $|w_0|< N$.  If $|w_0|<N$, the algorithm stops and outputs $w':=w_0\in\EE$.
If $|w_0|\ge N$,  then $ha^K\in G$ is not a proper power and the algorithm outputs $w':= w_0\tilde{a}^K$.  By construction,  
$p(w)=p(w')$,  and  $C_G(w')=\<w'\>$ if $w'\not\in\EE$. 
\end{proof}

\section{Proof of Theorem \ref{t:iff}}\label{s:proof}
 
We restate the theorem, for the reader's convenience.
 
\mainthm* 

\subsection{Choice of generators and notation} We fix epimorphisms $\pi_1:G_1\onto Q$ and $\pi_2:G_2\onto Q$
so that $$P = \{ (g_1,g_2) \mid \pi_1(g_1)=\pi_2(g_2)\} < G_1\times G_2.$$
All of the groups that we are dealing with are finitely generated,  so we are free to work with whatever generating
set we choose when establishing (un)decidability.  We begin by choosing a finite generating set $X_0$ for $Q$,
which defines an epimorphism from the free group $\mu_0:F(X_0)\onto Q$.  For $i=1,2$, we can factor $F(X_0)\onto Q$
through $\pi_i: G_i\onto Q$; let $\mu_i : F(X_0)\to G_i$ be this lift of $\mu_0$.  We choose a finite set $A_i\subseteq\ker \pi_i$
that generates $\ker \pi_i$ as a normal subgroup and is such that $A_1\cup\mu_i(X_0)$
generates $G_i$. We then define $X=X_0\sqcup A_1\sqcup A_2$ and for $i=1,2$ we extend $\mu_i$
by defining $\mu_i|_{A_i}$ to be the inclusion $A_i\hookrightarrow G_i$ while $\mu_1(A_2)=1$ and  $\mu_2(A_1)=1$. 
We  extend $\mu_0$ by defining $\mu_0(A_1\sqcup A_2)=1$. 
Thus we obtain compatible  sets of generators $\mu_i:F(X)\onto G_i$  for $G_i$ and $\mu_0:F(X)\onto Q$ for $Q$.
This compatibility will aid the transparency of the proof.

We suppress mention of the maps $\mu_i$,  writing expressions such
as ``$w=1$ in $Q$" and ``$u=v$ in $G_1$", for words $u,v,w\in F(X)$, when what we really mean is $\mu_0(w)=1$
and $\mu_1(u)=\mu_1(v)$.

We identify $G_1$ with $G_1\times 1< G_1\times G_2$ and  $G_2$ with $1\times G_2$.  Correspondingly\footnote{caution: with this notation, if $G_1=G_2$ but $\pi_1 \neq \pi_2$ then  $(x,x)$ is in $P$ but it might not be in the diagonal subgroup
of $G\times G$},
we have generators $(x,1)$ and $(1,x)$ for $G_1\times G_2$,  with  $x\in X$, but rather than working
with formal words in these symbols, we work with ordered pairs of words $(u,v)$ with $u,v\in F(X)$. 

Noting that $(a,1)=(a,a)$ and $(1,b)=(b,b)$ in $G_1\times G_2$ for each $a\in A_1\subset X$ and $b\in A_2\subset X$, 
 it is easily verified that $P$ is generated by $\{(x,x) \mid x\in X\}$.  

\subsection{The Proof of Theorem \ref{t:iff}}

We first prove $(2)\implies (1)$.
Given elements $U=(u_1,u_2)$ and $V=(v_1,v_2)$ 
of $P$,  we use the solution to the conjugacy problem in $G_i$
to decide if there exists words $w_1, w_2\in F(X)$   conjugating $u_1$ to $v_1$  in $G_1$
and $u_2$ to $v_2$ in $G_2$, respectively. If there is no such pair,
then $U$ is not conjugate to $V$. If such a pair does exist, then 
we  replace  $V$ by $(w_2,w_2)V({w_2},w_2)^{-1}$.
Thus we may assume, without loss of generality, that $v_2=u_2$ and that
 $w^{-1}u_1w=v_1$  in $G_1$ for some word $w\in F(X)$.

With this reduction, the set of elements of $G_1\times G_2$ conjugating $U$ to $V$ is
$$I=\{ (z_1^pw, z_2^q) \mid \<z_1\>=C_{G_1}(u_1),\ \<z_2\> = C_{G_2}(u_2)\}.$$ 
There is an algorithm to calculate the maximal roots of elements in a torsion-free hyperbolic group (Lemma \ref{l:find-root}), 
so we may assume that we have explicit words in the generators giving us $z_1$ and $z_2$ and
positive integers $e_i>0$ such that $z_i^{e_i}=u_i$.

To determine if $U$ is conjugate to $V$ in $P$,  we must decide whether $I\cap P$ is non-empty.    Since $(u_1,u_2)$ is in $P$,  
by multiplying on the left by a power of $(u_1,u_2)$, we can transform any
$\zeta\in  P\cap I$  into $(z_1^{p_1}w,1) (1,\e)\in I\cap P$ with $\e=z_2^j$ for some $0\le j<e_2$.
And $(z_1^{p_1}w,1) (1,\e)\in P$ if and only if $\e w^{-1} \in \<z_1\>$ in $Q$. Thus,  to determine whether
$I\cap P$ is non-empty, it suffices to check whether one of the finitely many elements $\e w^{-1}$ lies in
the cyclic subgroup of  $Q$ generated by $z_1$, which we can do using the solution to the power problem.
\smallskip

Next we prove $(1)\implies (2)$.  
Suppose that the conjugacy problem in $P$ is solvable.  First we claim that this hypothesis implies that the
word problem in $Q$ is solvable.  To help us see this, we first want to find  $g_i\in \ker(G_i\onto Q)$
with $C_{G_i}(g_i)=\<g_i\>$, for $i=1,2$.  To this end, 
we fix a non-trivial element $b_i\in \ker(G_i\onto Q)$.  For sufficiently large $p>0$,
the quotient $H_i:=G_i/\<\!\< b_i^p \>\!\>$ is hyperbolic (see \cite{olsh}, for example).  If the kernel 
of the  map $H_i\onto Q$ induced by $G_i\onto Q$ is finite, then $Q$ is hyperbolic and therefore has a solvable 
word problem. If $\ker (H_i\onto Q)$  is infinite,  then for each $N_i>0$ we can find an element  $c_i\in \ker(H_i\onto Q)$
with $d_{H_i}(1,c_i)>N_i$,  where $d_{H_i}$ is the word metric obtained by taking the image of $X$ as generators.  We want to 
apply Proposition \ref{t:no-power} with $a_i=b_i^p$. Let $K_i$ and $N_i$ be the constants of that proposition
(with subscripts to indicate which $G_i$ we are working with), choose
$c_i$ as above and,  following the proof of Proposition \ref{p:no-roots},  take a geodesic representative $v_i\in F(X)$ of $c_i$. 
Then $g_i:= v_ia_i^K\in G_i$ is an element of  $\ker(G_i\onto Q)$ that is not a proper power, therefore $C_{G_i}(g_i)=\<g_i\>$.

With $g_i$ in hand,  given $w\in F(X)$ we ask whether $(w^{-1} g_1 w,  g_2)\in P$ is conjugate to 
$(g_1,g_2)$ in $P$.   The elements of $G_1\times G_2$ conjugating $(g_1,g_2)$ to $(w^{-1} g_1 w,  g_2)$ are
$J=\{(g_1^pw, g_2^q)\mid p,q\in \Z\}$,  because $C_{G_i}(g_i)=\<g_i\>$.  And since $\<g_i\>\subseteq  \ker(G_i\onto Q)$,
the intersection $J\cap P$ will be non-empty if and only if $w\in \ker(G\onto Q)$.  Thus deciding whether
$(w^{-1} g_1 w,  g_2)$ is conjugate to  $(g_1,g_2)$ in $P$ tells us whether $w=1$ in $Q$.

In order to prove that the power problem is solvable in $Q$, we must exhibit an algorithm that,
given words $w,u\in F(X)$, will  determine whether of not $w\in \<{u}\>$ in $Q$. 
Using the solution to the word problem in $Q$,  
we first decide  whether or not $[w,u]=1$ in $Q$. 
If $[w,u]$ is non-trivial in $Q$,  then $w\not\in \<u\>$. If $[w,u]= 1$, then we proceed
as follows.  

First we use the solution to the word problem in $Q$ to decide if
$u\in Q$ is  in the image of the exceptional set $\EE$ of Proposition \ref{p:no-roots}. If $u$ is not in 
the image of $\EE$, then
use the algorithm from  Proposition \ref{p:no-roots} to replace $(u,u)\in P$ with $(u_1,u_2)\in P$
where  $u_i\in F(X)$ is such that $u=u_i$ in $Q$
and $C_{G_i}(u_i) = \<u_i\>$.   
 For any $g\in G_1$, the set of elements conjugating $(u_1,u_2)$ to $(g^{-1}u_1g, u_2)$ in $G_1\times G_2$
is then $\{ (u_1^pg, u_2^q) \mid p,q\in\Z\}$. 
Crucially, 
we know $(w^{-1}u_1w, u_2)$ is in $P$ because $[w,u]=1$ in $Q$. 
And deciding if $(u_1,u_2)$ is conjugate
to $(w^{-1}u_1w, u_2)$ in $P$ is equivalent to deciding if $(u_1^p w,u_2^q)\in P$ for some $p,q\in\Z$,
which is equivalent to deciding if ${w} \in \<u\>$ in $Q$. 
Thus, since the conjugacy problem in $P$ is solvable,  
we can decide whether or not  ${w} \in \<u\>$ in $Q$. 

It remains to consider what happens when $u\in Q$ is in the image of the finite set $\EE\subset F(X)$. In this case, 
we use the solution to the word problem in $Q$ to test which of the powers $u^2,u^3,\dots$ are  in the image of $\EE$.
We can assume that the empty word is in $\EE$, so we are simultaneously checking if $u^p=1$ in $Q$. Eventually, 
we will either determine that $u$ has finite order,  $m$ say, or else we will find $p>0$ such that $u^p\not\in\EE$.  If
we find $m=o(u)$, we use the solution to the word problem again to decide if $w=u^i$ in $Q$ for some 
$i\in\{0,\dots,m-1\}$.   If we find $u^p\not\in\EE$, then we use the algorithm of the previous paragraph to decide
whether $w\in \<u^p\>$ in $Q$.  If $w\in \<u^p\>$ then we have proved $w\in \<u\>$. If $w\not\in \<u^p\>$, then
for $i=1,\dots, p-1$ we apply the same algorithm with $wu^i$ in place of $w$ to decide if $wu^i\in \<u^p\>$,
noting that $w\in \<u\>$ if and only if the answer is yes in one case. 
\medskip

The implications $(2)\iff (3)$ were proved in Proposition \ref{ThmA:2iff3}, so the proof is complete.
\qed

\section{Statements and Declarations}
The author was not supported by any grants during the preparation of this manuscript.
The author has no relevant financial or non-financial interests to disclose. 

Data sharing is not applicable to this article, as no datasets were generated or analyzed during the current study.

 \end{document}